\documentclass[twoside,10pt]{article}%
\usepackage{amssymb}
\usepackage{amsfonts}
\usepackage{amsmath}
\usepackage{graphicx}
\usepackage{array}
\usepackage{multirow}
\usepackage{float}
\usepackage{varioref}

\providecommand{\U}[1]{\protect\rule{.1in}{.1in}}
 \arraycolsep=1.5pt 

\newtheorem{theorem}{Theorem}

\newtheorem{lemma}[theorem]{Lemma}

\newtheorem{proposition}[theorem]{Proposition}
\newtheorem{remark}[theorem]{Remark}

\newenvironment{proof}[1][Proof]{\noindent\textbf{#1.} }{\ \rule{0.5em}{0.5em}}
\begin{document}

\title{\textbf{A Prior Derivation and Local Existence of Classical Solutions for the Relativistic Euler Equations with Logarithmic Equation of State}}

\author{\textsc{Ka Luen Cheung},  \textsc{Sen Wong\thanks{Corresponding Author, E-mail address: senwongsenwong@yahoo.com.hk}} \ and  \textsc{Tat Leung Yee}\\\textit{Department of Mathematics and Information Technology,}\\\textit{The Education University of Hong Kong,}\\\textit{10 Lo Ping Road, Tai Po, New Territories, Hong Kong}}
\date{v25 Revised 12-September-2021}
\maketitle
\begin{abstract}
In this paper, from an investigation of a symmetric hyperbolic system, a prior derivation of the logarithm equation of state is provided. Through a diffeomorphism transforming the classical Euler equations to the symmetric hyperbolic system, we show that the logarithm pressure co-exists with existing barotropic equations of state without applying any physical laws. In this connection, we also establish a local existence of classical solutions for the relativistic Euler equations with the logarithm equation of state.

\


\

\end{abstract}
\newpage
\section{Introduction}
Originated from the local conservation of stress-energy for a perfect fluid in the $4$-dimensional Minkowski spacetime with metric signature $(-,+,+,+)$, when one fixes a space-time coordinates as $(t,x)=(t,x_1,x_2,x_3)$, the $3+1$-dimensional relativistic Euler equations (equation $(1.3)$ in $\cite{v1_r1}$) is expressed as
\begin{align}
    \left\{\begin{matrix}
    \partial_t\left(\displaystyle\frac{\rho c^2+p}{c^2-|v|^2}-\frac{p}{c^2}\right)+\nabla\cdot\left(\displaystyle\frac{\rho c^2+p}{c^2-|v|^2} v\right)=0,\\
    \partial_t\left(\displaystyle\frac{\rho c^2+p}{c^2-|v|^2}v\right)+\nabla\cdot\left(\displaystyle\frac{\rho c^2+p}{c^2-|v|^2}v\otimes v\right)+\nabla p=0,
    \end{matrix}\right.\label{v5_r1}
\end{align}
where $\rho\geq0$, $v$ and $p$ are the mass-energy density, transformed velocity and pressure respectively. $|v|$ is always less than the light speed $c$ so that $c^2-|v|^2$ is always positive.

The pressure $p$ in $(\ref{v5_r1})_2$ is determined by the state equation. In this circumstances, the state equation for logarithmic pressure, namely,
\begin{align}
p=A\ln\rho,\ A>0    \label{v18_e5}
\end{align}
is adopted here.

The logarithmic pressure $(\ref{v18_e5})$ has actually been introduced such as in \cite{v21_rr1, v21_rr2, v21_rr3} for the dark energy fluid to investigate the dynamical evolution of a late-time universe. When one compares $(\ref{v18_e5})$ with the polytropic equation of state, namely, $p=K_1 \rho^\gamma,\ \gamma>1,\ K_1>0$ and the equation of state for generalized Chaplygin gases, namely, $p=-K_2 \rho^{-\gamma},\ 0<\gamma\leq1,\ K_2>0$, it can be seen remarkably that the pressure $p(\rho)=A\ln\rho$ is positive when $\rho>1$ or negative when $0 <\rho< 1$. Thus, it can be regarded as a transition between the polytropic and generalized Chaplygin gases. As a consequence, it has been introduced as a natural and robust candidate for a new unification of dark matter and dark energy. For more mathematical properties of $(\ref{v18_e5})$, readers can refer to the last paragraph of section \ref{v48_s2}.

Mathematical results for both the classical or relativistic systems regarding the logarithmic equation of state is rare and mainly focus on properties of weak solutions of the Riemann problem. (see, for example, \cite{v25_rrr1} and the references therein). In particular, the local well-posedness of strong or classical solutions for system $(\ref{v5_r1})$ with logarithmic pressure has not been established. Thus, we aim to provide a steppingstone mathematical result on classical solutions of the relativistic Euler equation with the logarithmic pressure. More importantly, we argue the rationality for the logarithmic equation of state to exist by providing a prior derivation which shows that the logarithmic equation of state co-exists with the equations of state for the polytropic and generalized Chaplygin gases.

To close this section, we mention two recent finite-time blowup results for system $(\ref{v5_r1})$ with polytropic or generalized Chaplygin gases equation of state, among the other results in \\$\cite{v26_r3, v26_r4, v20_r3, v26_r6, v27_rr1, v26_r1, v26_r2, v20_r5, v5_r1,v27_rr2, v28.7_r1, v28.9_r1}$

When the equation of state is given by $p=p(\rho)$ such that i) $p(0)=0$, ii) $p(\rho)\geq0$ and iii) ${p}''(\rho)\geq0$ for $\rho\in(\rho_*,\rho^*)
$, where $0\leq\rho_*<\rho^*\leq\infty$, the authors in $\cite{v1_r1}$ proved a finite-time blowup result for the relativistic Euler equations, namely, system $(\ref{v5_r1})$, in the case of infinite energy when subluminal condition is adopted. More precisely, with the following subluminal condition,
\begin{align}
    0<{p}'(\rho)\leq c^2,\label{v20_e9}
\end{align}
where $c$ is the speed of light, the authors in $\cite{v1_r1}$ showed, in the spirit of Sideris $\cite{SI}$, that finite-time singularity for smooth solutions of system $(\ref{v5_r1})$ with initial data $(\ref{v11_e5})$ will be developed for small background energy-density $\bar{\rho}$ if the generalized mass is positive and the radial component of generalized momentum is large enough.

For the equation of state with negative pressure given by $p=-\frac{1}{\gamma}\rho^{-\gamma},\gamma\in(0,1]$, the authors in $\cite{r1_v18}$ considered the Cauchy problem of the $3+1$-dimensional relativistic Euler equations for generalized Chaplygin gas with non-vacuum initial data. It was shown that for large background energy-mass density and small pressure coefficient, the smooth solutions of the relativistic Euler equations for generalized Chaplygin gas with the generalized subluminal condition will blow up on finite time when the initial radial component of the generalized momentum is sufficiently large. Moreover, the blowup condition is independent of the signs of the generalized mass.
\section{A Prior Derivation of the Logarithmic Pressure}
In this section, a prior derivation of the logarithmic pressure is presented. 

Consider the following set of equations and ignore their physical meanings in the present context.
\begin{align}
    \left\{
\begin{array}
[c]{rl}%
\rho_t +\nabla\cdot(\rho u) & = 0,\\
\rho\lbrack u_{t}+(u\cdot\nabla)u] +\nabla p&= 0,
\end{array}
\right.\label{v45_e1}
\end{align}
where $\rho=\rho(t,x):[0,T)\times \mathbb{R}^N\to[0,\infty)$, $u=u(t,x):[0,T)\times\mathbb{R}^N\to\mathbb{R}^N$, $p=p(\rho)\in\mathbb{R}$ are the unknowns. Here $T\geq0$ is the life span of $C^1$ solutions of system $(\ref{v45_e1})$ and the fact that $T$ can be positive was established in $\cite{v35_r1}$.

Our assumptions imposed on $p(\rho)$ will be the following.
\begin{align}
    &p(\rho)\text{ is }C^2\text{ as a function of }\rho>0,\label{e6_v16}\\
    &{p}'(\rho)>0\text{ for all }\rho>0 \text{ and}\\
    &p(\rho)\text{ is either strictly convex or strictly concave}\text{ for all }\rho>0.\label{e8_v15}
\end{align}
Then, compare system $(\ref{v45_e1})$ with the following symmetric hyperbolic system:
\begin{align}
    \left\{
\begin{array}
[c]{rl}%
v_t+B\nabla\cdot u&=-u\cdot \nabla v-\displaystyle\frac{A}{2}v\nabla\cdot u,\\
u_t+B\nabla v&=-u\cdot\nabla u-\displaystyle\frac{A}{2}v\nabla v
\end{array}
\right.\label{v45_e2}
\end{align}
for some constants $A\neq0$ and $B\in\mathbb{R}$. Here, $u=u(t,x):[0,T_1)\times\mathbb{R}^N\to\mathbb{R}^N$ and $v=v(t,x):[0,T_1)\times\mathbb{R}^N\to\mathbb{R}$ with life span $T_1\geq0$ are the unknown $C^1$ functions. As a key result, the following proposition is reached.
\begin{proposition} Consider system $(\ref{v45_e1})$ with assumptions $(\ref{e6_v16})$ to $(\ref{e8_v15})$.

Then, the following map $v$ is a $C^1$ diffeomorphism from $\{\rho\ \in\mathbb{R} |\ \rho>0\}$ onto its image.
\begin{align}
    v:\rho\mapsto\frac{2}{A}\left(\sqrt{{p}'(\rho)}-B\right)\label{v51_e3},
\end{align}
where $B\in\mathbb{R}$ and $A\neq0$ are arbitrary constants; ${p}'(\rho)$ is the derivative of $p$ with respect to $\rho$. Moreover, system $(\ref{v45_e1})$ can be transformed through $v$ into (and hence is equivalent to) system $(\ref{v45_e2})$ with $T=T_1$ if and only if
\begin{align}
    p(\rho)=\left\{
    \begin{matrix}
    \displaystyle\frac{K_1}{A+1}\rho^{A+1}+K,\text{ for }A\neq1,\\
    K_1\ln\rho+K,\text{ for }A=1,
    \end{matrix}
    \right.\label{v45_e4}
\end{align}
where $K_1>0$ and $K$ are arbitrary constants.
\end{proposition}\label{v45_p1}
\begin{proof}
First, we prove that $v$ is a diffeomorphism.

Note that $v$ is differentiable with respect to $\rho$ and 
\begin{align}
    {v}'(\rho)=\frac{1}{A\sqrt{{p}'(\rho)}}{p}''(\rho).
\end{align}
From assumption $(\ref{e8_v15})$, we see that ${v}'(\rho)$ is either strictly positive or strictly negative for all $\rho$. Thus, $v$ is either strictly increasing or strictly decreasing. This proves $v$ is $C^1$ differentiable and injective. As the graph of $v^{-1}$ can be obtained by reflecting the graph of $v$ versus $\rho$ along the line $v=\rho$. One sees that $v$ is differentiable implies that $v^{-1}$ is also differentiable provided that ${v}'$ is never zero. After noting that ${(v^{-1})}'$ is equal to $1/{v}'$ by chain rule, one can conclude that $v$ is a $C^1$ diffeomorphism from $\{\rho\in\mathbb{R}\ |\ \rho>0\}$ onto its image. 

As ${p}'(\rho)\neq0$, one has, from $(\ref{v51_e3})$, that
\begin{align}
    v_t&=\frac{{p}''(\rho)}{A\sqrt{{p}'(\rho)}}\rho_t,\label{v57_e5}\\
    \nabla v&=\frac{{p}''(\rho)}{A\sqrt{{p}'(\rho)}}\nabla \rho,\\
    \sqrt{{p}'(\rho)}&=B+\frac{A}{2}v.\label{v57_e7}
\end{align}
Substituting $(\ref{v57_e5})$-$(\ref{v57_e7})$ into $(\ref{v45_e2})_1$, one sees that $(\ref{v45_e2})_1$ is equivalent to
\begin{align}
    \rho_t+\frac{A{p}'(\rho)}{{p}''(\rho)}\nabla\cdot u+u\cdot\nabla \rho=0.
\end{align}
Hence, $(\ref{v45_e2})_1$ is equivalent to $(\ref{v45_e1})_1$ if and only if
\begin{align}
    \frac{A{p}'(\rho)}{{p}''(\rho)}=\rho.\label{v57_e9}
\end{align}
Similarly, $(\ref{v45_e2})_2$ is equivalent to $(\ref{v45_e1})_2$ if and only if $(\ref{v57_e9})$ holds. As the full solutions of the differential equation $(\ref{v57_e9})$ is given by $(\ref{v45_e4})$ and the above steps are reversible, it follows that $T=T_1$ and the proposition is established.
\end{proof}

In the coming section, the significance of proposition $\ref{v45_p1}$ will be given. 
\section{Interpretation of Proposition $\ref{v45_p1}$}\label{v48_s2}
Note that the solutions in $(\ref{v45_e4})$ include the following two special cases.
\begin{align}
\text{Case 1: }A>0 \text{ and }A\neq1,\text{ then } p(\rho)=\frac{1}{\gamma}\rho^\gamma,\label{v46_e5}
\end{align}
where $\gamma:=A+1>1$
and
\begin{align}
\text{Case 2: }-2\leq A<-1,\text{ then } p(\rho)=-\frac{1}{\gamma}\rho^{-\gamma},\label{v46_e6}
\end{align}
where $\gamma:=-A-1\in(0,1]$.

Recall that system $(\ref{v45_e1})$ is the original $N$-dimensional compressible isentropic Euler equations in fluid mechanics, where $\rho$ and $u$ are the density and the velocity of the fluid respectively. $p=p(\rho)$ is the pressure function determined by a barotropic equation of state. (The term barotropic means that the pressure can be expressed as a function of the density.)

The first equation in $(\ref{v45_e1})$ is derived from the mass conservation law while the second equation in $(\ref{v45_e1})$ is a result of the momentum conservation law. 

For Euler equations in their primitive form, $p$ is given by the gamma law:
\begin{align}
    p=\frac{1}{\gamma}\rho^\gamma, \quad \gamma\geq1.\label{v46_e7}
\end{align}
This classical Euler equations are one of the most fundamental equations in fluid dynamics. Many interesting fluid dynamic phenomena can be described through system $(\ref{v45_e1})$ whose pressure is determined by the gamma law $\cite{v25_r1,v25_r2}$. Moreover, system $(\ref{v45_e1})$ with pressure $(\ref{v46_e7})$ is also the special case of the noted Navier-Stokes equations, whose problem of whether there is a formation of singularity is still open and long-standing.

On the other hand, when the state equation is given by
\begin{align}
    p=-\frac{1}{\gamma}\rho^{-\gamma},\quad 0<\gamma\leq1,\label{v46_e8}
\end{align}
system $(\ref{v45_e1})$ is called the Euler equations for generalized Chaplygin Gas (GCG).

Since the theory of dark matter and dark energy were developed in the literature such as $\cite{v52_r1, v52_r2, v52_r3}$ for explaining that the mass of ``visible" matter only comprise $4\%$ of the total energy density in the four-dimensional standard cosmology model, the GCG model has been introduced in $\cite{v52_r5}$ and developed in $\cite{v52_r4}$ as a unification of dark matter and dark energy, where the invisible energy component is regarded as a unified dark fluid. i.e. a mixture of dark matter and dark energy.

The key is that one observes that the derivations of $(\ref{v46_e5})$ $\&$ $(\ref{v46_e6})$ and $(\ref{v46_e7})$ $\&$ $(\ref{v46_e8})$ are independent. Hence, the state equations $(\ref{v46_e7})$ with $\gamma>1$ and $(\ref{v46_e8})$ originated from fluid mechanic and cosmology are \textbf{rediscovered} from the equivalent transformation $(\ref{v45_e2})$ without predetermined physical knowledge. Moreover, in addition to $(\ref{v46_e5})$ and $(\ref{v46_e6})$, there is an extra outcome in $(\ref{v45_e4})$, namely, the logarithmic pressure:
\begin{align}
    p=K_1\ln\rho+K.\label{v46_e9}
\end{align}

Without loss of generality, one may set $K=0$. Below, we add some observable mathematical proprieties of the logarithmic pressure $(\ref{v46_e9})$.

First, it takes both positive and negative values while $(\ref{v46_e7})$ and $(\ref{v46_e8})$ take only positive and negative values respectively. Second, as $\ln \rho$ is undefined when $\rho=0$, one has to consider non-vacuum initial data for the system. (The same is true for $(\ref{v46_e8})$). Third, the two ``tails" of $(\ref{v46_e9})$ coincide with $(\ref{v46_e7})$ and $(\ref{v46_e8})$ in the sense that $\lim_{\rho\to\infty}\ln\rho=\lim_{\rho\to\infty}\rho^{\gamma}$ for $\gamma\geq1$ and $\lim_{\rho\to 0^+}\ln\rho=\lim_{\rho\to0^+}(-\rho^{-\gamma})$ for $0<\gamma\leq1$. Fourth, $(\ref{v46_e9})$ satisfies the requirement that ${p}'(\rho)=\frac{K_1}{\rho}>0$ which is the square of sound speed and is required to be positive by a fundamental thermodynamic assumption.

\begin{remark}
In $\cite{v47_r1}$, the authors established a shock formation result for the two dimensional (i.e. $N=2$) system $(\ref{v45_e1})$ with state equation $(\ref{v46_e7})$. Within the arguments, a change of variable of $\rho$, namely, $\ln\rho$ was introduced. We remark that $(\ref{v46_e9})$ is not a change of variable but an introduction of possible state equation with potential applications and physical interpretation for the general $N$-dimensional system $(\ref{v45_e1})$. 
\end{remark}

\section{Further Assumptions on $p$}
In order to establish the local existence result, we impose two assumptions on the logarithmic equation of state, $p=A\ln\rho$. For the rationality of assuming the common necessary convention, namely, the subluminal condition, readers can refer to section $2$ in \cite{r1_v18}.
\begin{align}
    &\text{Assumption \#1: the subluminial condition: }{p}'(\rho)\leq c^2,\label{e26_v7}\\
    &\text{Assumption \#2: }A\text{ is large enough such than }A> c^2/e,\label{e28_v7}\\
\end{align}
where $e$ is the Euler constant for natural logarithm.

We first present the implications of the above assumptions.
\begin{lemma}\label{v25_l3}
For $p=A\ln\rho$, an implication of Assumption \#1 is that $\rho$ has a uniform positive lower bound:
\begin{align}
&\rho\geq\rho_*:=\frac{A}{c^2}>0\label{e27_v7}    
\end{align}
and
an implication of Assumption \#1 together with Assumption \#2 is that the ratio of $p$ to $\rho$ has a uniform negative lower bound:
\begin{align}
    \frac{p}{\rho}>-c^2.\label{27_v16}
\end{align}
\end{lemma}
\begin{proof}
By definition, we have
\begin{align}
{p}'(\rho)=A/\rho&\leq c^2,
\end{align}
which implies $(\ref{e27_v7})$.

Next, to prove $(\ref{27_v16})$, it suffices to show $\rho c^2+p>0$:
\begin{align}
    \rho c^2+p&=\rho c^2+A\ln\rho\\
    &\geq\frac{A}{c^2}c^2+A\ln\frac{A}{c^2}\\
    &=A+A\ln\frac{A}{c^2}\\
    &> A+A\ln \frac{1}{e}\\
    &=0.
\end{align}
Thus, $(\ref{27_v16})$ is also established.
\end{proof}
\section{Local Existence of Classical Solutions}

Consider the Cauchy problem of system $(\ref{v5_r1})$ with the following initial data.
\begin{align}
    \rho(x,0)=\rho_0(x)>0,\ v(x,0)=v_0(x),\label{v11_e5}
\end{align}
The precise statement of the local existence result is presented as follows.
\begin{theorem}\label{t10_v7}
Under the assumptions $\#1$ and $\#2$ and given initial data $(\rho_0,v_0)$ belonging to the uniformly local Sobolev space $H_{ul}^s=H_{ul}^s(\mathbb{R}^3)$ with $s\geq3$, and that there exists a sufficiently small positive constant $\delta$ such that
\begin{align}
    \rho_*+\delta\leq\rho,\text{ and } |v_0|^2\leq(1-\delta)c^2
\end{align}
hold for all $x\in\mathbb{R}^3$,
then the Cauchy problem of system $(\ref{v5_r1})$ with pressure $(\ref{v18_e5})$ and initial data $(\ref{v11_e5})$ has a unique solution
\begin{align}
    (\rho,v)\in L^\infty([0,T];H_{ul}^s)\cap C([0,T];H_{loc}^s)\cap C^1([0,T];H_{loc}^{s-1})
\end{align}
for some $T=T(\delta,\rho_0,v_0)>0$. Here, $|*|$ is the usual Euclidean norm.

\end{theorem}
\begin{proof}
The proof is divided into $4$ steps.

\textbf{Step 1}: Outline of Proof.

We first present a scheme for the proof of theorem $\ref{t10_v7}$.

In step 2, we will construct a map
\begin{align}
    (\rho,v)=(\rho,v_1,v_2,v_3)^{\top}\mapsto w:=(w_0,w_1,w_2,w_3)^{\top}\label{e45_v7}
\end{align}
such that system $(\ref{v5_r1})$ can be transformed into the following form
\begin{align}
    A^0(w)w_t+\sum_{k=1}^3A^k(w)w_{x_k}=\mathbf{0},\label{e46_v7}
\end{align}
where $A^k$, $k=0,1,2,3$ are $4$ by $4$ coefficient matrices depending on $w$. Here, the symbol $\top$ denotes the transpose operation of matrices.

In step 3, we shall show that the coefficient matrices satisfy i) $A^k$, $k=0,1,2,3$ are all real symmetric and smooth in $w$ and ii) $A^0$ is positive definite. Then, system $(\ref{e46_v7})$ is a symmetric hyperbolic system and its local existence will be followed by Kato's Theorem \cite{v5_r3, v21_rr4}. 

Last, in step 4, we shall prove that the map $(\ref{e45_v7})$ is a diffeomorphism to complete the proof of theorem $\ref{t10_v7}$. 

\textbf{Step 2}: Transformation to Symmetric Hyperbolic System

Inspired by Pan and Smoller \cite{v1_r1}, whose construction could be traced back to Makino \cite{v5_r2, v20_r4}, we claim the map is given by
\begin{align}
    \left\{\begin{matrix}
    w_0:=-\displaystyle\frac{c^3Ke^{\phi(\rho)}}{(\rho c^2+p)\sqrt{c^2-|v|^2}}+c^2,\\
    w_k:=\displaystyle\frac{cKe^{\phi(\rho)}}{(\rho c^2+p)\sqrt{c^2-|v|^2}}v_k,\ k=1,2,3,
    \end{matrix}\right.\label{e32_v12}
\end{align}
where
\begin{align}
    \phi(\rho):=\int_{\rho_*}^{\rho}\frac{c^2}{\rho c^2+p(\rho)}d\rho
\end{align}
and
\begin{align}
    K:=c^2\rho_*+p(\rho_*).
\end{align}
Note that the definition of $\rho_*$ is given by $(\ref{e27_v7})$.

Write $A^k(w)=(A^k_{ij})$, $i,j=0,1,2,3$. After a tedious but direct computation, one can show that system $(\ref{v5_r1})$ will be transformed into $(\ref{e46_v7})$ with
\begin{align}
    A^0=\Psi(\rho)\begin{pmatrix}
    B_1&B_2v_1&B_2v_2&B_2v_3\\
    B_2v_1&B_3v_1v_1+B_4&B_3v_1v_2&B_3v_1v_3\\
    B_2v_2&B_3v_2v_1&B_3v_2v_2+B_4&B_3v_2v_3\\
    B_2v_3&B_3v_3v_1&B_3v_3v_2&B_3v_3v_3+B_4
    \end{pmatrix}
\end{align}
and
\begin{align}
    (A^k)_{ij}=\left\{\begin{matrix}
    \Psi(\rho)B_2&\ \text{if }i=j=0,\\
    \Psi(\rho)(B_3v_jv_k+B_5\delta_{jk})&\ \text{if }i=0\text{ and }j\neq0,\\
    \Psi(\rho)(B_3v_iv_k+B_5\delta_{ik})&\ \text{if }i\neq0\text{ and }j=0,\\
    \Psi(\rho)\left[B_3v_iv_jv_k+B_4(v_i\delta_{jk}+v_j\delta_{ik}+v_k\delta_{ij})\right]&\ \text{if }i\neq0\text{ and }j\neq0,
    \end{matrix}\right.
\end{align}
where
\begin{align}
    \Psi(\rho)&:=\frac{(\rho c^2+p)^2}{c^3K{p}'e^{\phi(\rho)}(c^2-|v|^2)^{3/2}},\label{e37_v11}\\
    B_1&:=c^4+3{p}'|v|^2,\\
    B_2&:=c^4+2c^2{p}'+{p}'|v|^2,\\
    B_3&:=c^2(c^2+3{p}'),\\
    B_4&:=c^2{p}'(c^2-|v|^2),\\
    B_5&:=\frac{c^2{p}'(c^2-|v|^2)}{\rho c^2+p}.\label{e42_v11}
\end{align}
Note that $(\ref{e37_v11})$ to $(\ref{e42_v11})$ are always positive by lemma $\ref{v25_l3}$.

\textbf{Step 3}: Verification of Conditions of Kato's Theorem

It is clear that $A^k$, $k=0,1,2,3$ are all smooth real symmetric matrices. We shall show that $A^0$ is positive definite as follows.

Write
\begin{align}
    A^0&=\begin{pmatrix}
    \mathcal{A}&\mathcal{B}\\
    \mathcal{B}^\top&\mathcal{C}
    \end{pmatrix},
\end{align}
where
\begin{align}
    \mathcal{A}&:=B_1>0,\\
    \mathcal{B}&:=B_2v^\top,\\
    \mathcal{C}&:=B_3vv^\top+B_4I_{3\times 3}.
\end{align}

By Schur Complement Method, it suffices to show
\begin{align}
    \mathcal{C}-\mathcal{B}^\top \mathcal{A}^{-1}\mathcal{B}=(B_3-B_1^{-1}B_2^2)vv^\top+B_4I_{3\times 3}\label{e55_v9}
\end{align}
is positive definite.

However, the three real eigenvalues of $(\ref{e55_v9})$ are given by 
\begin{align}
    \lambda_1&=\lambda_2=B_4>0,\\
    \lambda_3&=(B_3-B_1^{-1}B_2^2)|v|^2+B_4=\frac{{p}'(c^2-|v|^2)^2(c^4-{p}'|v|^2)}{c^4+3{p}'|v|^2}>0.
\end{align}
This implies that $A^0$ is positive definite

\textbf{Step 4}: Returning to the Original System.

In this step, we shall show that the map $(\ref{e45_v7})$ is a diffeomorphism by Inverse Function Theorem.

Let us denote the map $(\ref{e45_v7})$ by $\mathbf{w}$. The domain of $\mathbf{w}$ is
\begin{align}
    \Omega:=\{(\rho,v)\in\mathbb{R}^4\ |\ \rho_*< \rho\text{ and }|v|^2<c^2\},
\end{align}
where $\rho_*=A/c^2$.

We also denote
\begin{align}
    \Phi(\rho):=\frac{K e^{\phi(\rho)}}{\rho c^2+p(\rho)}.
\end{align}
Note that $\Phi(\rho)>0$ by lemma $\ref{v25_l3}$ and
\begin{align}
    {\Phi}'(\rho)=\frac{-K{p}'(\rho)e^{\phi(\rho)}}{(\rho c^2+p)^2}<0.\label{e52_v14}
\end{align}
By direct computation, the Jacobian of $\mathbf{w}$ is given by
\begin{align}
    \left(\frac{\partial \mathbf{w}}{\partial \rho },\frac{\partial \mathbf{w}}{\partial v_1 },\frac{\partial \mathbf{w}}{\partial v_2 },\frac{\partial \mathbf{w}}{\partial v_3 }\right)=\frac{c\Phi}{(c^2-|v|^2)^{3/2}}\begin{pmatrix}
    -c^2(c^2-|v|^2)({\Phi}'/ \Phi)&-c^2v^\top\\
    (c^2-|v|^2)({\Phi}'/ \Phi) v&vv^\top+(c^2-|v|^2)I_{3\times 3}.
    \end{pmatrix}
\end{align}
and its determinant is
\begin{align}
    \left(\frac{c\Phi}{(c^2-|v|^2)^{3/2}}\right)^4\times\left(-(c^2-|v|^2)^4c^2\frac{{\Phi}'}{\Phi}\right)&=-\frac{c^6\Phi^3(\rho)}{(c^2-|v|^2)^2}{\Phi}'(\rho)>0
\end{align}
for all $\rho$ and $v$.

Thus, the Jacobian determinant of $\mathbf{w}$ is always non-zero. By Inverse Function Theorem, $\mathbf{w}$ is a local diffeomorphism.

Next, we shall show that $\mathbf{w}$ is a global diffeomorphism from $\Omega$ onto its image $\mathbf{w}(\Omega)$ by showing the injectivity of $\mathbf{w}$.

Note that we have
\begin{align}
    |v|^2=\frac{c^4}{(c^2-w_0)^2}(w_1^2+w_2^2+w_3^2).\label{e52_v12}
\end{align}
The proof of $(\ref{e52_v12})$ is as follows.
\begin{align}
    w_1^2+w_2^2+w_3^2&=\left(\displaystyle\frac{cKe^{\phi(\rho)}}{(\rho c^2+p)\sqrt{c^2-|v|^2}}\right)^2\left(v_1^2+v_2^2+v_3^2\right)\\
    &=\left(\frac{c\Phi(\rho)}{\sqrt{c^2-|v|^2}}\right)^2|v|^2.\label{e54_v12}
\end{align}
On the other hand, we have
\begin{align}
    (c^2-w_0)^2&=\left(\displaystyle\frac{c^3Ke^{\phi(\rho)}}{(\rho c^2+p)\sqrt{c^2-|v|^2}}\right)^2\\
    &=c^4\left(\frac{c\Phi(\rho)}{\sqrt{c^2-|v|^2}}\right)^2.\label{e56_v12}
\end{align}
Thus, $(\ref{e52_v12})$ is followed by $(\ref{e54_v12})$ and $(\ref{e56_v12})$. The proof of $(\ref{e52_v12})$ is complete. 

Next, we show that $\Phi(\rho)$ can be expressed as
\begin{align}
    \Phi(\rho)&=\frac{1}{c^2}\left((c^2-w_0)^2-c^2(w_1^2+w_2^2+w_3^2)\right)^{1/2}.\label{e57_v12}
\end{align}
Note that, from $(\ref{e54_v12})$ and $(\ref{e56_v12})$, we have
\begin{align}
    (c^2-w_0)^2-c^2(w_1^2+w_2^2+w_3^2)&=c^4\left(\frac{c\Phi(\rho)}{\sqrt{c^2-|v|^2}}\right)^2-c^2\left(\frac{c\Phi(\rho)}{\sqrt{c^2-|v|^2}}\right)^2|v|^2\\
    &=c^4\Phi^2(\rho)\left(\frac{c^2}{c^2-|v|^2}-\frac{|v|^2}{c^2-|v|^2}\right)\\
    &=c^4\Phi^2(\rho).
\end{align}
As $\Phi(\rho)>0$, the result follows. The proof of $(\ref{e57_v12})$ is complete.

Recall from $(\ref{e52_v14})$ that we have
\begin{align}
    {\Phi}'(\rho)=\frac{-K{p}'(\rho)e^{\phi(\rho)}}{(\rho c^2+p)^2}<0.
\end{align}
It follows that $\Phi$ is strictly decreasing and hence injective.

Now, we claim that the injectivity of $\Phi(\rho)$ implies the injectivity of $\mathbf{w}$. The proof is as follows.

Suppose the images of $(\rho_1,v_{11},v_{12},v_{13})$ and $(\rho_2,v_{21},v_{22},v_{23})$ coincide under $\mathbf{w}$. That is, 
\begin{align}
\mathbf{w}(\rho_1,v_{11},v_{12},v_{13}):=(w_{10},w_{11},w_{12},w_{13})=(w_{20},w_{21},w_{22},w_{23})=:\mathbf{w}(\rho_2,v_{21},v_{22},v_{23})
\end{align}
We want to show
\begin{align}
    (\rho_1,v_{11},v_{12},v_{13})=(\rho_2,v_{21},v_{22},v_{23}).
\end{align}
From $(\ref{e57_v12})$, we have
\begin{align}
    \Phi(\rho_1)=\Phi_(\rho_2).
\end{align}
Hence, $\rho_1=\rho_2$ as $\Phi$ is injective.

On the other hand, from $(\ref{e52_v12})$, we have
\begin{align}
|v_{1*}|^2:=v_{11}^2+v_{12}^2+v_{13}^2=v_{21}^2+v_{22}^2+v_{23}^2=|v_{2*}|^2.    
\end{align}
Recall that from $(\ref{e32_v12})_2$, we have
\begin{align}
    w_{jk}&=\displaystyle\frac{cKe^{\phi(\rho)}}{(\rho c^2+p)\sqrt{c^2-|v_{j*}|^2}}v_{jk}\\
    &=\frac{c\Phi(\rho_j)}{\sqrt{c^2-|v_{j*}|^2}}v_{jk},\ j=1,2;\ k=1,2,3.
\end{align}
As
\begin{align}
    \frac{c\Phi(\rho_1)}{\sqrt{c^2-|v_{1*}|^2}}=\frac{c\Phi(\rho_2)}{\sqrt{c^2-|v_{2*}|^2}},
\end{align}
we have $w_{1k}=w_{2k}$ implies $v_{1k}=v_{2k}$, $k=1,2,3$. The proof of the injectivity of $\mathbf{w}$ is complete.

The proof of theorem $\ref{t10_v7}$ is complete by replacing $\rho_*$ by $\rho_*-\delta>0$ for any small $\delta>0$.

\end{proof}
\section{Acknowledgement}
This research is partially supported by the 2021-22 SRG of the MIT Department, and the Internal Research Grant RG8/2021-2022R, the Education University of Hong Kong.

\end{document}